\documentclass[12pt,reqno,twoside]{amsart}
\usepackage{amsmath}
\usepackage{amssymb}
\usepackage{epsfig}
\usepackage
{color}

\usepackage[margin=1in]{geometry}

\numberwithin{equation}{section}

\title[A Zero-One Law for {improvements  to Dirichlet's Theorem}]{{A Zero-One Law for 
{improvements \\ to Dirichlet's Theorem}}}
\author{Dmitry Kleinbock}
\address{Brandeis University, Waltham MA
02454-9110 {\tt kleinboc@brandeis.edu}}
\author{Nick Wadleigh}
\address{Brandeis University, Waltham MA 02454-9110
{\tt wadleigh@brandeis.edu}}
\newif\ifdraft\drafttrue
\draftfalse

\newcommand\eq[2]{{\ifdraft{\ \tt [#1]}\else\ignorespaces\fi}\begin{equation}\label{eq:
#1}{#2}\end{equation}}

\newcommand {\equ}[1] {\eqref{eq: #1}}

\newcommand{\Q}{{\mathbb {Q}}}

\newcommand{\R}{{\mathbb{R}}}
\newcommand{\Z}{{\mathbb{Z}}}

\newcommand{\N}{{\mathbb{N}}}

\newcommand{\vv}{{\bf{v}}}

\newcommand{\vb}{{\bf{b}}}

\newcommand{\SL}{\operatorname{SL}}

\newcommand{\diag}{{\operatorname{diag}}}

\newcommand{\x}{{\bf x}}

\newcommand{\vp}{{\bf p}}

\newcommand{\vq}{{\mathbf{q}}}

\newcommand{\nz}{\smallsetminus\{0\}}

\newcommand\hd{Hausdorff dimension}

\newcommand\da{Diophantine approximation}

\newcommand\hs{homogeneous space}

\newcommand\dt{Dirichlet's Theorem}

\newcommand\mr{M_{m,n}
}
\newcommand\amr{$Y\in M_{m,n}
$}
\newcommand {\ignore}[1] {}

\newcommand\ssm{\smallsetminus}

\newtheorem{thm}{Theorem}[section]
\newtheorem{lem}[thm]{Lemma}
\newtheorem{que}[thm]{Question}
\newtheorem{prop}[thm]{Proposition}
\newtheorem{cor}[thm]{Corollary}

\newtheorem{remark}[thm]{Remark}

\begin{document}
\ignore{}
\begin{abstract}
{We give an integrability condition on a function $\psi$ guaranteeing that for almost all (or almost no) {$x\in\R$}, the system $|qx-p|< \psi(t)$, $|q|<t$  is solvable in $p\in \mathbb{Z}$, $q\in \mathbb{Z}\ssm \{0\}$  for sufficiently large $t$. Along the way, we characterize such $x$ in terms of the growth of their continued fraction entries, and we establish that Dirichlet's Approximation Theorem is sharp in a very strong sense. {Higher-dimensional generalizations are discussed at the end of the paper.}}
\end{abstract}
\thanks{The first-named author was supported by NSF grants DMS-1101320 and DMS-1600814.}

\date{January 19, 2017}
\maketitle
\section{Introduction and motivation}\label{intro}
\noindent The starting point for the present paper, as well as for numerous endeavors in the theory of \da, is the following theorem, established by Dirichlet in 1842:
\begin{thm}[Dirichlet's Theorem]\label{dirichlet theorem} For any $x\in \mathbb{R}$ and $t> 1$, there exist $q \in \mathbb{Z}\nz,\ p \in \mathbb{Z}$ such that
\eq{dt}{ |qx - p|
\le \frac 1t
\ \ \ \textit{and}
\ \ |q|
< t.}
\end{thm} 

\noindent 
{See e.g.\ \cite[Theorem I.I]{Cassels} or \cite[Theorem I.1A]{Schmidt}.} In many cases {the above} theorem has been applied through its corollary  {\cite[Corollary I.1B]{Schmidt}, 
predating Dirichlet's work}:
\begin{cor}[Dirichlet's Corollary] \label{dirichlet corollary} For any $x\in \mathbb{R}$ there exist infinitely many $q \in \mathbb{Z}$ such that
\eq{dc}{ |qx - p|{\,<\,} \frac1{|q|} \quad \text{for some }p\in\Z.}
\end{cor}
\noindent The two statements above give a rate of approximation which works for all $x$ and serve as a beginning of the {\sl metric theory of \da}, which is concerned with understanding sets of $x$ satisfying conclusions similar to those of Theorem \ref{dirichlet theorem} and Corollary \ref{dirichlet corollary} with
the right hand sides of \equ{dt} and \equ{dc} replaced by faster decaying functions of $t$ and $|q|$ respectively.
Those sets are very well studied in the setting of Corollary \ref{dirichlet corollary}. Indeed, for a function {$\psi: [t_0,\infty) \to \R_+$, where $t_0\ge 1$ is fixed,} let us define
${W}(\psi)$, the set of {\sl $\psi$-approximable} real numbers,
to be the set of $x\in\R$ for which
there exist infinitely many $q\in\Z$ such that
\eq{dcpsi}{ |qx - p|< \psi(|q|) \quad \text{for some }p\in\Z.}

In what follows we will use the notation $\psi_1(t) = 1/t$.
Then Corollary \ref{dirichlet corollary} asserts that ${W}(\psi_1) = \mathbb{R}$. It is well known that there exists $c > 0$ such that ${W}(c\psi_1) \ne \mathbb{R}$. {{In fact,} numbers which do not belong to ${W}(c\psi_1) $ for some $c > 0$ {(equivalently,  irrational numbers whose continued fraction coefficients are uniformly bounded)} are called {\sl badly approximable}.  
 It is known that such numbers form a set of full \hd\  {\cite{Jarnik}}.
However the Lebesgue measure of the set of badly approximable numbers is zero; in other words, ${W}(c\psi_1)$ is co-null for any $c > 0$. Precise conditions for the Lebesgue measure of ${W}(\psi)$ to be zero or full are given by
\begin{thm}[Khintchine's Theorem]\label{Khintchine theorem}
Given a non-increasing $\psi$,
the set ${W}(\psi)$ has zero (resp.\ full) measure if and only if the series $\sum_k\psi(k)$ converges (resp.\ diverges). \end{thm}
\noindent Quite surprisingly, {it seems that }no such clean statement has yet been proved in the set-up of Theorem \ref{dirichlet theorem}. This is the aim of the present paper. 

\smallskip
We start by introducing the following definition: { for $\psi$ as above, let $D(\psi)$ denote the set of $x\in \mathbb{R}$ for which 
the system 
\eq{dtpsi}{ |qx - p|
< \psi(t)
\ \ \ \mathrm{and}
\ \ |q|
< t}
has a nontrivial integer solution for all large enough $t$.  Elements of $D(\psi)$ will be called $\psi$-$Dirichlet$.
Notice that this definition arises by replacing} ``$\le \psi_1(t)$'' in \equ{dt} with ``$ <\psi(t)$'' and demanding the existence of nontrivial integer solutions for all $t$ except those belonging to a bounded set. {Here are some elementary observations:}
{
\begin{itemize}
\item If $\psi$ is non-increasing, which will be our standing assumption, one can without loss of generality restrict to $t\in\N$: indeed, to solve \equ{dtpsi}  it is enough to find a solution with $t$ replaced by $\lceil t \rceil$.
\item It is not hard to see 
that $D(\psi_1) = \mathbb{R}$; {more precisely, if $x\notin\Q$ (resp.\ if $x\in\Q$), the system \equ{dtpsi} with $\psi = \psi_1$  has a nonzero solution for all $t > 1$ (resp.\  for sufficiently large $t$).} \item Clearly $D(\psi)$ is contained in ${W}(\psi)$ whenever $\psi$ is non-increasing.
\end{itemize}} 
{On the other hand, one knows that} $D(\psi)$ and $W(\psi)$ differ significantly for functions $\psi$ decaying faster than {$\psi_1$}. For example, it has been observed by Davenport and Schmidt \cite{Davenport-Schmidt} that the set $D(c\psi_1)$ of $c\psi_1$-Dirichlet numbers has Lebesgue measure zero for any $c < 1$. Moreover, they showed {\cite[Theorem 1]{Davenport-Schmidt}} that {an irrational} number belongs to $D(c\psi_1)$  for some $c < 1$ if and only if it is 
badly approximable. {Thus $x \in \cup_{c<1} D(c\psi_1)$  if and only if 
the continued fraction coefficients of $x$ are uniformly bounded.} This naturally motivates the following questions:

\begin{que} \label{cf question}
Can {one} characterize $x\in D(\psi)$ in terms of its continued fraction expansion? \end{que} 

{\begin{que}\label{Dirichlet sharp} Is Dirichlet's theorem sharp in the sense that if 
{$\psi(t)<\psi_1(t)$ for {all} sufficiently large $t$},  then {there exists $x\in\R$ which is not $\psi$-Dirichlet?}
 \end{que} }

\begin{que} \label{dirichlet question}
What is a necessary and sufficient condition on 
$\psi$ (presumably, expressed in the form of convergence/divergence of a certain series) guaranteeing that
the set $D(\psi)$ has zero/full measure? \end{que} 

\noindent 
In this paper we answer {Questions   \ref{cf question} and \ref{Dirichlet sharp}} in the affirmative and give {an answer} to 
{Question   \ref{dirichlet question}} under {an additional} assumption that {the function $t\mapsto t\psi(t)$}  is non-decreasing.  Specifically, we will prove the following:
{\begin{thm} \label{Sharpness} If $\psi:[t_0,\infty)\rightarrow \R_+$ is non-increasing and $\psi(t)<\psi_1(t)$ for sufficiently large $t$, then $D(\psi)\neq \R$. \end{thm}}
\begin{thm} \label{Main Theorem} Let $\psi: [t_0,\infty) \rightarrow \R_+$ be non-increasing, and suppose the function $t\mapsto t\psi(t)$ is non-decreasing {and \eq{lessthan1}{t\psi(t) < 1\quad\text{for all }t\ge {t_0}.}}  Then if \eq{condition}{
{\sum_n\frac{-\log\big(1-n\psi(n)\big)\big(1-n\psi(n)\big)}{n}}
= \infty \hspace{3mm}(\text{resp.} <\infty),}
then the Lebesgue measure of $D(\psi)$ (resp.\  of $D(\psi)^c$) is zero.\end{thm}

\noindent {We note that \equ{lessthan1} is a natural assumption: if it is not satisfied, then $D(\psi) = {D(\psi_1) =\, }\R$.


{ {As an example},  taking  $\psi = c\psi_1$ 
in Theorem \ref{Main Theorem}  makes the 
{sum} in \equ{condition} {equal to}
$$
{\sum_n\frac{-c\log(1-c)}n}
\,.$$ Thus we recover the aforementioned} result of Davenport and Schmidt stating that $D(c\psi_1)$ has measure zero for $c<1$.  
{Here are two more examples\footnote{The functions  below are not non-increasing, but \textit{eventually} decreasing; clearly only the eventual behavior of $\psi$ is relevant.}:}

\smallskip
\begin{itemize}
\item
{If $\psi(t)=\frac{1-at^{-k}}{t}$  for $a > 0$, $k\geq0$, then  the 
{sum} in \equ{condition} 
{converges/diverges if and only if so does}
$$\int_1^\infty \frac{-\log (t^{-k})t^{-k}}{t} \, dt= \int_1^\infty \frac{k\log t }{t^{k+1}} \, dt. $$Thus $D(\psi)
$ has full measure whenever $k>0$.}  
\smallskip

\item {If $\psi(t)= \frac{1-a(\log t)^{-k}}{t}$ for $a > 0$, $k\geq0$, we are led to consider $$\int_e^\infty \frac{k\log\log t }{t(\log t)^{k}}\, dt = \int_1^\infty \frac{k\log u}{u^{k}} \, du.$$ In this case $D(
\psi)$ has full measure if $k>1$ and zero measure otherwise.}
\end{itemize}

\smallskip
 {The structure of the paper is as follows.  Theorem  \ref{Sharpness} is proved 
in the next section, following some lemmas expressing the $\psi$-Dirichlet property via continued fractions. In  \S\ref{main proof}  we discuss dynamics of the Gauss map $x\mapsto \frac{1}{x}- \lfloor\frac{1}{x}\rfloor$ in the unit interval and, following \cite{Philipp}, establish 
a dynamical Borel-Cantelli lemma. {This Borel-Cantelli lemma is then used to prove Theorem  \ref{Main Theorem}.} In the last section of the paper we discuss possible higher-dimensional generalizations.}

\medskip

\noindent { {\bf Acknowledgements}. The authors are grateful to   Mumtaz Hussain and Bao-Wei Wang for helpful discussions and to an anonymous referee for useful comments.}

\section{Continued fractions}\label{cf}

 \noindent Denote by $\langle x\rangle$ the distance from $x$ to the nearest integer. Throughout the sequel, $a_n=a_n(x)$ ($n=1,2,...$) will denote the $n$th entry in the continued fraction expansion {of $x\in [0,1)$,} and $q_n=q_n(x)$ will refer to the denominator of the $n$th convergent to $x$. That is
 $$[a_1(x),a_2(x),...,a_n(x)]= \frac{1}{a_1+\frac{1}{a_2+\frac{1}{...+\frac{1}{a_n}}}}=  \frac{p_n}{q_n}.$$
with $p_n$, $q_n \in \N$ coprime. If we take $q_0=1$,  $\{ q_n \}_{n=0}^\infty$ may be defined as the increasing sequence of positive integers with the property $\langle q_nx \rangle< \langle qx \rangle$ for all positive integers $q<q_n$.   The sequences $\{a_n\}$, $\{q_n\}$ are 
{related by the recurrence \eq{recurrence}{q_n=a_nq_{n-1}+q_{n-2}.}} 
  We refer the reader to \cite{Khintchine} 
 or the first chapter of [Ca] for background on the theory of continued fractions\footnote{Note however that { Cassels' }definition of the sequence $\{q_n\}$ differs from that of Khintchine by one index. That is, Cassels has $q_1=1, q_2=a_1,...$ .  We have adopted Khintchine's notation.}.  

\smallskip

We prefer to work with $x$ for which the sequences $q_n(x)$, $a_n(x)$ do not terminate; that is,  
{we} exclude {the case $x\in\Q$}.  Since all the properties that concern us are invariant under translation by $\Z$, we 
{will only consider} $x\in[0,1)\ssm \Q$.

\begin{lem}\label{convergents}  Let $\psi: [{t_0},\infty)\rightarrow {\R_+}$ be non-increasing.  {Then} $x\in [0,1]\ssm \Q$ is $\psi$-Dirichlet if and only if $\langle q_{n-1}x\rangle< \psi (q_n)$ for sufficiently large $n$.  \end{lem}

\begin{proof}  Suppose  $x\in [0,1]\ssm \Q$ 
is $\psi$-Dirichlet.  
Then for sufficiently large $n$ there exists a positive integer, $q$, with $\langle qx\rangle < \psi(q_n)$, $q<q_n$.  Since  $\langle q_{n-1}x\rangle\leq \langle qx\rangle$ whenever $q<q_n$, we have $\langle q_{n-1}x\rangle < \psi(q_n)$ for sufficiently large $n$.  Conversely, suppose $\langle q_{n-1}x\rangle< \psi (q_n)$ for $n\geq N$.  Then for a real number $t>q_N$, write $q_{n-1}<t\leq q_n$.  The {inequality} $\langle q_{n-1}x\rangle< \psi(t)$ {follows} since $\psi$ is non-increasing.   Thus $x$ is $\psi$-Dirichlet.
\end{proof}

Lemma \ref{convergents} is one step toward rephrasing {the $\psi$-Dirichlet property} 
of $x$ in terms of the growth of the continued fraction entries, $a_n(x)$.  
 {For fixed  {$x=[a_1,a_2,...]$,} consider the sequences $$\theta_{n+1}=[a_{n+1}, a_{n+2}, ...], \quad \phi_n= [a_n, a_{n-1},..., a_1].$$  These are related to the sequences $q_n$, $\langle q_{n-1}x\rangle$ by the identity   
\eq{q2a}{(1+\theta_{n+1}\phi_n)^{-1}= q_n\langle q_{n-1}x\rangle}
(see \cite[\S II.2]{Cassels}\footnote{In truth, Cassels' formula reads  $(1+\theta_{n+1}\phi_n)^{-1}= q_{n+1}\langle q_nx\rangle$ because his $q_n$'s are shifted by one index, as we have already noted.}).   This is our device for passing from Lemma \ref{convergents} to continued fractions, 
allowing us to answer  Question \ref{cf question}.}

\begin{lem}\label{label1}  Let  {$x\in [0,1]\ssm \Q$,} and let $\psi:[{t_0},\infty)\rightarrow \R_+$ be non-increasing with $t\psi(t)<1$ for all $t\geq t_0$.  Then

\begin{itemize}
\item[(i)]  $x$ is $\psi$-Dirichlet if $a_{n+1}a_n \leq\dfrac{1}{4} 
\left(\big(q_n\psi(q_n )\big)^{-1}-1\right)^{-1}$ for all sufficiently large $n$.

\item[(ii)]  $x$ is \textit{not} $\psi$-Dirichlet if $a_{n+1}a_n > \left(\big(q_n\psi(q_n )\big)^{-1}-1\right)^{-1}
$ for infinitely many $n$.
\end{itemize}
\end{lem}

\begin{proof}  Fix  {$x\in [0,1]\ssm \Q$.}  Using \equ{q2a}, Lemma \ref{convergents} becomes  \eq{label}{x\in D(\psi) \textit{ if and only if }  (1+\theta_{n+1}\phi_n)^{-1} < q_n\psi(q_n) \textit{ for all large n.}}
Since $( a_{n+1}+\frac{1}{a_{n+2}})(a_n+\frac{1}{a_{n-1}})\leq 4a_{n+1}a_n$, we have
\begin{equation*}
\begin{aligned} \left(1+\frac{1}{a_{n+1}}\cdot \frac{1}{a_n}\right)^{-1}& < (1+\theta_{n+1}\phi_n)^{-1} \\ &<
 \left(1+\frac{1}{a_{n+1}+\frac{1}{a_{n+2}}}\cdot \frac{1}{ a_n+\frac{1}{a_{n-1}}}  \right)^{-1}\leq \left(1+\frac{1}{4a_{n+1}a_n}\right)^{-1} .
 \end{aligned}
 \end{equation*}
Hence from \equ{label}, $x \in D(\psi)$ if 
$$\left(1+\frac{1}{4a_{n+1}a_n}\right)^{-1}\leq q_n\psi(q_n)$$
for sufficiently large $n$. We get the first assertion of the lemma by solving for $a_na_{n+1}$.
Similarly, $x\notin D(\psi)$ if
$$\left(1+\frac{1}{a_{n+1}a_n}\right)^{-1}\ > q_n\psi(q_n)$$
for unbounded $n$.  Solving for $a_na_{n+1}$ gives the second assertion of the lemma.  \end{proof}

{\begin{remark}\label{} Lemma \ref{label1} generalizes the aforementioned result of Davenport and Schmidt \cite[Theorem 1]{Davenport-Schmidt} stating that $x\in D(c\psi_1)$ for some $c<1$ if and only if the sequence $a_n(x)$ is uniformly bounded.
\end{remark}}

{Now we can} answer Question \ref{Dirichlet sharp} {and exhibit real numbers which are not $\psi$-Dirichlet for any non-increasing $\psi$ with $\psi(t)<\psi_1(t)$ for sufficiently large $t$}.




\begin{proof}[Proof of Theorem \ref{Sharpness}]
By the recurrence \equ{recurrence}, $q_n$ depends only on $a_1,..., a_n$.   Since $t\psi(t)<1$  { for large enough $t$}, we may construct  $x=[a_1,a_2,...]$ by successively choosing $a_{n+1}$ so that part (ii) of Lemma \ref{label1} is satisfied.  \end{proof}
     
 {\begin{remark} {We point out} that for a given $\psi$, the proof of Theorem \ref{Sharpness} is entirely constructive, since each $q_n$ is determined recursively by the preceding choice of $a_n$.  Also note that the proof constructs $x$ such that the system \equ{dtpsi} is insoluble when $t=q_n$ for {\em all sufficiently large} $n$ -- not just  for infinitely many $q_n$.  \end{remark}}

\section{{Borel-Cantelli Lemmas}}\label{main proof}

For almost every $x$, we have reduced {the $\psi$-Dirichlet property of $x$}  to the growth of {its} continued fraction entries.  The \textit{Gauss map},
\eq{Gauss Map}{T: [0,1]\ssm \mathbb{Q} \rightarrow [0,1] \ssm \mathbb{Q}, \hspace{3mm} x\mapsto x^{-1}-\lfloor x^{-1}\rfloor}
 has the convenient property $T([a_1, a_2, a_3, ...])=[a_2, a_3, a_4....]$, and it preserves the \textit{Gauss measure},
\eq{Gauss Measure}{\mu(A)= \frac{1}{\log 2}\int_A\frac{1}{1+x}\,dx.}


{
{We will use} two {results of Philipp \cite{Philipp}}
related to the mixing rate of $T$ and the divergence case of the Borel-Cantelli Lemma.  

\begin{thm}{\cite[Theorem 2.3]{Philipp}}\label{Ph1}. Let $E_n$, $n\geq 1$, be a sequence of measurable sets in a probability space  $(X, \nu)$.  Denote by $A(N,x)$ the number of integers $n\leq N$ such that $x\in E_n$.  Put $$\phi(N)=\sum_{n\leq N} \nu(E_n).$$  Suppose that there exists a convergent series $\sum_{j\geq 1} C_j$ with $C_j \geq 0$ such that for all integers $m>n$ we have\eq{mixingineq}{\nu(E_n \cap E_m) \leq \nu(E_n)\nu(E_m)+ \nu(E_m)C_{m-n}.} Then {for any  $\epsilon >0$ one has} $$A(N,x)= \phi(N)+ O_\epsilon\big(\phi^{1/2}(N)\log^{3/2+\epsilon} \phi(N)\big) 
$$ for almost all x.
\end{thm}

 {\begin{remark}\label{+z} Since $\left| \nu(E_n \cap E_m) - \nu(E_n)\nu(E_m)\right|\le \nu(E_m)$, Theorem \ref{Ph1} can be trivially strengthened: given any ${\ell}>0$, the conclusion of 
 {the theorem} holds provided the inequality \equ{mixingineq} holds whenever $m>n+{\ell}$ .  \end{remark}}

\begin{thm}\label{Ph2} {\cite[Theorem 3.2]{Philipp}}  There exist constants $c_0 > 0$ and $ 0 < {\gamma} < 1$ with the following property. Fix $\mathbf{r} = (r_1,...r_k) \in \N^k$, and write
 $$E_\mathbf{r}:= \{x \in [0,1] \ssm \Q: a_1(x) =r_1, \ a_2(x)=r_2, 
\ \dots\ ,\  a_k(x)=r_k \}.$$ 
Let $F\subset {[0,1]}$ be any measurable set.  Then for all $n\geq 0$,
\eq{overlap}{\left|\mu(E_{\mathbf{r}}\cap T^{-n-k}F) - \mu(E_{\mathbf{r}}) \mu(F)\right| \le c_0\mu(E_{\mathbf{r}})\mu(F){\gamma}^{\sqrt{n}}.}  
 \end{thm}} 
 
 {As Philipp observed, this estimate admits passing to unions:}

{\begin{cor}\label{UnionMixing} Let  {$c_0
$ and $
{\gamma}
$} be as in Theorem \ref{Ph2}. Let $F\subset {[0,1]}$ be any measurable set.  Fix $k\in \N$, and let $\mathcal{R}\subset \N^k$. Then \equ{overlap} holds for all $n\geq 0$ when $E_\mathbf{r}$ is replaced with $\cup_{\mathbf{r}\in \mathcal{R}}E_{\mathbf{r}}$.
\end{cor}

{\begin{proof} We have
\begin{equation*}
\begin{aligned}
\left|\mu(\cup_{\textbf{r}\in \mathcal{R}} E_\mathbf{r} \cap T^{-n-k}F) - \mu(\cup _{\textbf{r}\in \mathcal{R}} E_\mathbf{r}) \mu(F)\right| &= \left| \sum_{\mathbf{r}\in \mathcal{R}} \mu(E_\mathbf{r}\cap T^{-n-k}F)-\mu(E_\mathbf{r})\mu(F)\right| \\ &\le \sum_{\mathbf{r} \in \mathcal{R}} c_0\mu(E_\mathbf{r})\mu(F){\gamma}^{\sqrt{n}}= c_0\mu(\cup_{\mathbf{r}\in \mathcal{R}}E_\mathbf{r})\mu(F){\gamma}^{\sqrt{n}}.
\end{aligned}
\end{equation*} \end{proof}  

{We now combine the above statements to establish a quite general dynamical Borel-Cantelli lemma:}

\begin{lem}\label{lemma 2}  Fix $k\in \N$. Suppose $A_n$ $(n\in \mathbb{N})$ is a sequence of sets 
{such} that each $A_n$ is a union of sets of the form $E_{\mathbf{r}}$, $\mathbf{r}\in \N^k$ ($E_\mathbf{r}$ as defined in Thereom \ref{Ph2}). If $\sum_n \mu(A_n)=\infty$ (resp.\ $< \infty$), then {for} almost every (resp.\ almost no) $x\in [0,1] 
$ {one} has $T^n (x)\in A_n$ for infinitely many $n$. \end{lem} \begin{proof} The convergence case follows from the Borel-Cantelli Lemma and the fact that $\mu$ is $T$-invariant.  {Suppose $\sum_n \mu(A_n) =\infty$. For $m\geq n+k$ write} 
$$\mu(T^{-n}A_n\cap T^{-m}A_m)= \mu(A_n \cap T^{-(m-n)}A_m)\leq \mu(A_n)\mu(A_m)+c_0\mu(A_m)\mu(A_n){\gamma}^{\sqrt{m-n-k}}$$
$$\leq \mu(A_n)\mu(A_m)+\mu(A_m)c_0{\gamma}^{\sqrt{m-n-k}}=\mu(T^{-n}A_n)\mu(T^{-m}A_m)+\mu(T^{-m}A_m)c_0{\gamma}^{\sqrt{m-n-k}} $$
\noindent for ${c_0,\gamma}$
as in Theorem \ref{Ph2}. 
 The sets $T^{-n}A_n$ therefore satisfy the condition of  {Theorem} \ref{Ph1} {(in light of Remark \ref{+z})}.  By that theorem, $\sum_n\mu(T^{-n}A_n)=\infty$ guarantees that almost all $x$ lie in $T^{-n} A_n$ for infinitely many $n$.  \end{proof}

{The above lemma can now be applied to describe real numbers which belong to infinitely many sets of the form $ \{x: a_1(x)a_2(x)>{\Psi}(n)\}$:}

\begin{thm}\label{lemma 3} Let ${\Psi}: \mathbb{N}\rightarrow [1,\infty)$ be any function with $\lim_{n\rightarrow \infty}{\Psi}(n)=\infty$. If $$\sum_n \frac{\log {\Psi}(n)}{{\Psi}(n)}<\infty\text{\  (resp.\ $=\infty$),}$$ then almost every (resp.\ almost no) $x\in [0,1] \ssm \Q$ has {\eq{of n}{a_{n+1}(x)a_n(x)\leq {\Psi}(n)}} for sufficiently large $n$. \end{thm}

\begin{proof}  Define
$$A_n:= \{x: a_1(x)a_2(x)>{\Psi}(n)\}=  \bigcup_{a=1}^\infty \bigcup_{b=\lfloor \frac{{\Psi}(n)}{a}+1\rfloor}^\infty \bigg(\frac{1}{a+\frac{1}{b}}, \frac{1}{a+\frac{1}{b+1}}\bigg)= \bigcup_{a=1}^\infty \bigg(\frac{1}{a+\frac{1}{\lfloor \frac{{\Psi}(n)}{a}+1\rfloor}}, \frac{1}{a}\bigg).$$
Clearly $x\in [0,1]\ssm \mathbb{Q}$ has $a_{n+1}a_n> {\Psi}(n)$ if and only if $T^{n-1}(x)\in A_n$, where $T$ denotes the Gauss map \equ{Gauss Map}. By Lemma \ref{lemma 2}, it suffices to show $$c^{-1}\mu(A_n) \leq \frac{\log {\Psi}(n) }{{\Psi}(n)} \leq c\mu (A_n)$$ for some $c>0$ for all large $n$.  In fact, since $\frac{\lambda}{\log 2}\leq \mu \leq \frac{2\lambda}{\log 2}$, where $\lambda$ is Lebesgue measure on [0,1], it suffices to show $$c^{-1}\lambda(A_n) \leq \frac{\log{\Psi}(n)}{{\Psi}(n)} \leq c\lambda (A_n).$$

 We have
$$A_n \subset \left( \bigcup_{a\leq {\Psi}(n)}\bigg(\frac{1}{a+\frac{a}{{\Psi}(n)}}, \frac{1}{a}\bigg) \right) \bigcup \left(\bigcup_{a>{\Psi}(n)}\bigg(\frac{1}{a+1}, \frac{1}{a}\bigg)\right)\subset$$
$$ \left( \bigcup_{a\leq {\Psi}(n)}\bigg(\frac{1}{a+\frac{a}{{\Psi}(n)}}, \frac{1}{a}\bigg) \right) \bigcup \bigg(0, \frac{1}{{\Psi}(n)}\bigg)=$$
$$\bigg(\frac{1}{1+\frac{1}{{\Psi}(n)}}, 1\bigg)\bigcup \left( \bigcup_{a=2}^{{\lfloor\Psi(n)\rfloor}}\bigg(\frac{1}{a+\frac{a}{{\Psi}(n)}}, \frac{1}{a}\bigg) \right) \bigcup \bigg(0, \frac{1}{{\Psi}(n)}\bigg).$$
So 
\begin{equation*}
\begin{aligned}\lambda(A_n)&\leq 1-\frac{1}{1+\frac{1}{{\Psi}(n)}} +\int_1^{{\Psi}(n)}{\left(\frac{1}{a}-\frac{1}{a+\frac{a}{{\Psi}(n)}}\right)} da+ \frac{1}{{\Psi}(n)}\\&=
1-\frac{1}{1+\frac{1}{{\Psi}(n)}}+ \log {\Psi}(n) \left(1-\frac{1}{1+\frac{1}{{\Psi}(n)}}\right)  +\frac{1}{{\Psi}(n)}\\&=
 \frac{1}{{\Psi}(n)+1}+\frac{\log{\Psi}(n)}{1+{\Psi}(n)}+\frac{1}{{\Psi}(n)}\asymp \frac{\log {\Psi}(n) }{{\Psi}(n)}\,.
\end{aligned}
\end{equation*}

\smallskip

 To see the asymptotic lower bound, we start with
$$A_n\supset \bigcup_{a=1}^{{\lfloor\Psi(n)\rfloor}} \bigg(\frac{1}{a+\frac{1}{\frac{{\Psi}(n)}{a}+1}}, \frac{1}{a}\bigg).$$
Then

\begin{equation*}
\begin{aligned}\lambda(A_n)&\geq \int_1^{{\Psi}(n)} {\left(\frac{1}{a}-\frac{1}{a+\frac{1}{\frac{{{\Psi}(n)}}{a}+1}}\right)}da =\int_1^{{\Psi}(n)} \frac{1}{a(a+{{\Psi}(n)}+1)}da\\
&=\int_1^{{\Psi}(n)} {\left(\frac{({{\Psi}(n)}+1)^{-1}}{a}-\frac{({{\Psi}(n)}+1)^{-1}}{{{\Psi}(n)}+a+1}\right)}da= \frac{\log{{\Psi}(n)}}{{{\Psi}(n)}+1}+\frac{\log(\frac{{{\Psi}(n)}+2}{2{{\Psi}(n)}+1})}{{{\Psi}(n)} +1}\asymp \frac{\log{{\Psi}(n)}}{{{\Psi}(n)}}.
\end{aligned}
\end{equation*}
\end{proof}

{Comparing Theorem \ref{lemma 3} with Lemma \ref{label1}, one can see that in order to answer Question  \ref{dirichlet question}, one would need to replace the right hand side of \equ{of n} with a function depending on $q_n$. This can be easily achieved using known facts about the growth of  $q_n(x)$ for almost all $x$.}


{\begin{cor}\label{with qn} Let ${\Psi}: \mathbb{N}\rightarrow [1,\infty)$ be a non-decreasing function with $\lim_{n\rightarrow \infty}{\Psi}(n)=\infty$. If \eq{newsum}{\sum_{n}\frac{\log{\Psi}(n)}{n{\Psi}(n)}<\infty\text{\  (resp.\ $=\infty$),}} then almost every (resp.\ almost no)  $x\in [0,1] \ssm \Q$
 has $a_{n+1}(x)a_n(x)\leq {\Psi}\big(q_n(x)\big)$ for sufficiently large $n$. 
\end{cor}

\begin{proof}
 {There exists { $ b>1$} such that {
\eq{b}{\text{for 
every }  x\notin\Q ,  \quad b^n \leq q_n(x)\text{ for all  }n\ge 2}}
(see \cite[\S4]{Khintchine}). There also exists $B>{ b}$ such that 
{\eq{B}{\text{for almost 
every }  x, \quad q_n(x)\leq B^n \text{ for all large enough } n}}(see \cite[\S14]{Khintchine}). }  By using Cauchy's condensation argument  it is straightforward to see that  \eq{codiv}{\sum_{n}\frac{\log{\Psi}(n)}{n{\Psi}(n)}= \infty  \quad \iff \quad \sum_{n}\frac{\log{\Psi}(b^n)}{{\Psi}(b^n)}=\infty \quad \iff \quad \sum_{n}\frac{\log{\Psi}(B^n)}{{\Psi}(B^n)}=\infty.}
Thus if the sum in \equ{newsum}   converges,  Theorem \ref{lemma 3} implies that  
almost every $x\in [0,1] \ssm \Q$ has $$a_{n+1}(x)a_n(x)\leq {\Psi}(b^n) \underset{\equ{b}}\leq {\Psi}\big( q_n(x) \big)$$ for sufficiently large $n$. Conversely, if the sum in \equ{newsum}   diverges, \equ{codiv} and Theorem \ref{lemma 3} imply that for  almost every $x\in [0,1] $, one has $$a_{n+1}(x)a_n(x) >{\Psi}(B^n) \underset{\equ{B}}{\ge} {\Psi}\big(q_n(x)\big)$$  for infinitely many $n$. 
\end{proof}

{\begin{remark}\label{lemma2.5} {The proof of Corollary \ref{with qn} also  shows that, modulo a null set,  it is possible to describe  $\psi$-Dirichlet points in a way similar to Lemma \ref{label1}, but with the bounds on $a_{n+1}a_n$ depending on $n$ and not on $q_n$. Namely, with $b,B$ as above, almost every $x$ is $\psi$-Dirichlet if $a_{n+1}a_n\leq \frac{1}{4} \cdot ([b^n\psi(b^n )]^{-1}-1)^{-1}$ for all sufficiently large $n$, and is  not $\psi$-Dirichlet if $a_{n+1}a_n> ([B^n\psi(B^n )]^{-1}-1)^{-1}$ for infinitely many $n$. {Here we use the hypothesis of Theorem \ref{Main Theorem} that $t\mapsto t\psi(t)$ is non-decreasing.}
}\end{remark}}

We are now ready to characterize $\psi$ such that $D(\psi)$ has zero/full measure.

\medskip

\begin{proof}[Proof of Theorem \ref{Main Theorem}]  If $t \psi(t)$ is bounded away from $1$, $D(\psi)$ is null since $D(c\psi_1)$ is null for any $c<1$ [DS1].  We therefore assume $t\psi(t) \rightarrow 1$ as $t\rightarrow \infty$ (recall $t\psi(t)$ is assumed non-decreasing).  Let us write ${\Psi}(t):=(1-t\psi(t))^{-1}$.  The 
{sum} in Theorem \ref{Main Theorem} becomes  
\eq{sum}{\sum_{n}\frac{\log{\Psi}(n)}{n{\Psi}(n)}.}  Note that this sum converges if and only if it converges when ${\Psi}(n)$ is replaced with $c{\Psi}(n)$ for any $c>0$. Also note that ${\Psi}(n)$ is asymptotic to the function that appears in Lemma \ref{label1}. That is, \eq{asymp}{{\Psi}(n) \asymp \left( \big(n\psi(n)\big)^{-1}-1\right)^{-1} ~~~ \text{as} ~~~n \rightarrow \infty.}

{Suppose the sum \equ{sum} converges. {Then by Corollary \ref{with qn}, for any $\varepsilon>0$, almost every $x$ has $$a_{n+1}(x)a_n(x) \leq \varepsilon {\Psi}\big(q_n(x)\big)$$ for all large enough $n$. Thus 
Lemma \ref{label1}(i)} and the limit \equ{asymp} imply that $D(\psi)$ has full measure.}
{Conversely, suppose {that} \equ{sum} diverges.  {Then for any $M>0$, almost every $x$ has $$a_{n+1}(x)a_n(x) > M {\Psi}\big(q_n(x)\big)$$ for infinitely many $n$.
Therefore 
Lemma \ref{label1}(ii)} and the limit \equ{asymp} imply that $D(\psi)$ has {measure zero}.}\end{proof}

\section{{
{Generalizations to higher dimensions}}}\label{concl}


Let $m,n$ be positive integers, and denote by $\mr$ 
the space of $m\times n$ matrices with real entries. 
{The following is the general form of}  \dt\ 
on simultaneous Diophantine approximation (see  e.g.\ \cite[\S I.5]{Cassels} or \cite[Theorem II.1E]{Schmidt}):
{ \begin{thm}\label{dirichlet theorem general} { For any $Y \in M_{m,n}$
and }$t > 1$ there exist
$\vq = (q_1,\dots,q_n) \in \Z^n\nz$ and 
$\vp = (p_1,\dots,p_m)\in \Z^m$ satisfying the following system of inequalities:
\eq{dtgeneral}{
\|Y\vq - \vp\|
\le t^{-n/m}
  \ \ \ \mathrm{and}  
\ \ \|\vq\|
{\,<\,} t
\,.}
\end{thm}
}
\noindent Here 
$\|\cdot\|$ stands for the norm on $\R^k$ given by 
$\|\x\| = \max_{1\le i \le k}|x_i|$.  

\smallskip
Let $\psi:[t_0,\infty)\rightarrow \R_+$ be non-increasing.  {In analogy} 
 with the definition for $m = n = 1$, 
 let us say that  $Y\in\mr$ is {\sl $\psi$-Dirichlet}, 
and write $Y\in {D}(\psi)$, 
if 
for every sufficiently
large $t$  one can find
$\vq  \in \Z^n\nz$ and $\vp \in \Z^m$
 with
\eq{digeneral}{
\|Y\vq - \vp \|^m
< \psi(t)
  \ \ \ \mathrm{and}  
\ \ \|\vq\|^n
<  t
.}
{Note that}  the sharpness result of Davenport and Schmidt mentioned in the introduction also holds in higher dimensions: 
  {the Lebesgue measure of ${D}(c\psi_{1})$ is zero for any $c < 1$. 
See \cite[Theorem 1]{Davenport-Schmidt2} for the case $\min(m,n) = 1$, and \cite[Theorem 4]{KW} for further generalizations.  This naturally motivates higher-dimensional analogues of Questions \ref{Dirichlet sharp} and
 \ref{dirichlet question}:} 
  


{\begin{que}\label{Dirichlet sharp2} {Is Theorem \ref{dirichlet theorem general} sharp in the sense that if $\psi$ is non-increasing and \linebreak {$\psi(t)<\psi_{1}(t)$ for all sufficiently large $t$},  then {there exists \amr\ which is not $\psi$-Dirichlet?}}
 \end{que} 
 
\begin{que} \label{dirichlet question2}
For fixed 
$m,n\in\N$, what is a necessary and sufficient condition on a non-increasing 
$\psi$ (presumably, expressed in the form of convergence/divergence of a certain series) guaranteeing that
the set ${D}(\psi){\subset \mr}$ has zero/full measure? \end{que} 

In higher dimensions the machinery of continued fractions is no longer available.  It is nonetheless still possible to restate the problem in terms of a shrinking target phenomenon in a dynamical system.
{This approach is based on ideas from \cite{Davenport-Schmidt2} and \cite{Dani}, and, in a more explicit form -- on 
\cite[\S8]{KM}, 
where}
the Khintchine-Groshev theorem (the natural higher dimensional analogue of Theorem \ref{Khintchine theorem}) is proved using a dynamical Borel-Cantelli Lemma for a diagonal flow on the space of unimodular lattices in $\R^{m+n}$. The starting point for the reduction is the ``Dani Correspondence":

\begin{lem}\label{dani1}{\cite[Lemma 8.3] {KM}}  {Fix $m,n\in\N$ and $t_0\ge 1$, and let $\psi : [t_0, \infty)\rightarrow \R_+$ be} a continuous, non-increasing function.
Then there exists a unique continuous function $${r
: [s_0,\infty)\rightarrow \mathbb{R},\text{ where }s_0 = \frac m{m+n}\log t_0 - \frac m{m+n}\log \psi(t_0),}$$ such that the function $s\mapsto s-nr(s)$ is strictly increasing and unbounded, the function $s\mapsto s+mr(s)$ is nondecreasing, and \eq{identity}{\psi(e^{s-nr(s)})=e^{-s-mr(s)} ~~\textit{for all} ~ s\geq {s_0}. }  
 \end{lem}

{Denote by $X$  the space of unimodular lattices in $\R^{m+n}$}, and define $$\Delta:  X  \rightarrow \mathbb{R},\  \Lambda \mapsto -\log \inf_{\vv\in \Lambda \ssm \{0\}} \|\vv\|.$$  
{$X\cong \SL_{m+n}(\R)/ \SL_{m+n}(\Z)$ is a noncompact \hs. According to Mahler's Compactness Criterion, a subset  $K$ of $X$ is relatively compact if and only if the restriction of $\Delta$ to $K$ is bounded from above. Also, in view of Minkowski's Lemma, $\Delta$ is always bounded from below by $0$. Furthermore, 
\eq{K0}{K_0:=\Delta^{-1}(0)} is a union of finitely many compact submanifolds of $X$, whose structure is explicitly described by Haj\'os-Minkowski Theorem (see \cite[\S XI.1.3]{CasselsGN} or \cite[Theorem 2.3]{Shah}).} 

\smallskip
For $Y \in M_{m,n}$, define 
$$\Lambda_Y:=  \left( \begin{array}{cc} I_m & Y \\
0 & I_n \\ \end{array} \right) \mathbb{Z}^{m+n} \in X.$$  Finally, define $${g}_s
:= \diag(e^{s/m},...,e^{s/m},e^{-s/n},...,e^{-s/n}),$$
where there are $m$ copies of $e^{s/m}$ and $n$ copies of $e^{-s/n}$.   
We may now 
{rephrase the $\psi$-Dirichlet property of \amr} as a statement about the orbit of $\Lambda_Y$ in the dynamical system $(X, {g}_s)$:

\begin{prop}\label{homogeneousshrinkingtarget} Fix positive integers $m, n$, and let $\psi: [t_0, \infty)\rightarrow {\R_+}$ be continuous, non-increasing  {and such that} $\psi(t)<1$ for large enough $t$. Let $r=r_\psi$ be as in Lemma \ref{dani1}. Then $Y\in {D}(\psi)$ if and only if $$\Delta ({g}_s\Lambda_Y)>  r_\psi(s)$$ for all sufficiently large $s$ \end{prop}

\begin{proof} {Recall that
 $Y\in {D}(\psi)$  if and only if for large enough $t$ the system \equ{digeneral} has a solution $(\vp,\vq)$ with $\vq  \in \Z^n\nz$ and $\vp \in \Z^m$. 
 If $\psi(t) < 1$, all solutions $(\vp,\vq)\ne 0$ to this system will have $\vq \ne 0$. Since $\psi$ is eventually less than $1$, $Y\in {D}(\psi)$  if and only if  \equ{digeneral} is solvable in $(\vp,\vq)\in \Z^{m+n}\nz$ for sufficiently large $t$.}
 
Since {the function $s\mapsto s-nr(s)$} is increasing and unbounded, $Y\in {D}(\psi)$ if and only if for large enough $s$, 
$$\|Y\vq-\vp\|^m<  \psi (e^{s-nr(s)})=e^{-s-mr(s)} \hspace{10mm} \|\vq\|^n <  e^{s-nr(s)}$$
for some $(\vp, \vq) \in \mathbb{Z}^{m+n}\nz$.  This is equivalent to 
$$e^{s/m} \|Y\vq-\vp \| < e^{-r(s)}\hspace{10mm} e^{-s/n} \|\vq\| < e^{-r(s)}, $$
which is the same as $\Delta({g}_s \Lambda_Y)> r_\psi(s)$.\end{proof}


Thus $Y\notin {D}(\psi)$ if and only if ${g}_s\Lambda_Y\in \Delta^{-1}\left(\big[0, r_\psi(s){\big]}\right)$ for {an} unbounded {set of} $s\in\R_+$.   
{For example, the choice $\psi = c\psi_1$ for $c < 1$ in view of  \equ{identity} yields $$r(s) \equiv r_c :=\frac1{m+n} \log(1/c),$$ a constant function. That is, $Y\notin {D}(c\psi_1)$ if and only if ${g}_s\Lambda_Y\in \Delta^{-1}\left(\big[0, r_c{\big]}\right)$ for {an} unbounded {set of} $s\in\R_+$.  Therefore the aforementioned fact that ${D}(c\psi_1)$ is null for any $c< 1$ follows from the ergodicity of the $g_s$-action on $X$ and the set $\{\Lambda_Y : Y\in\mr\}$ being an unstable leaf for this action.}

{In general,}  the targets $\Delta^{-1}{\big([0,r]\big)}$ are 
neighborhoods of the set {$K_0$ as in \equ{K0}}.
We are 
{thus} interested in whether these shrinking targets are hit at an unbounded {set of} times by {trajectories of} a measure-preserving flow. 
 There are some technical obstructions, perhaps surmountable, to this approach to Questions  \ref{Dirichlet sharp2} {and \ref{dirichlet question2}}. However, in a forthcoming paper \cite{KWa}  we use {a similar} approach to solve {an} analogous inhomogeneous problem. Specifically, we 
 {establish} a dynamical Borel-Cantelli Lemma for the flow ${g}_s$ on the space of \textit{affine} {unimodular} lattices {in $\R^{m+n}$}, and go on to prove the following result:

\begin{thm} Let $\psi: [{t_0},\infty ) \rightarrow {\R_+}$ be 
non-increasing. 
 If  $$
 {\sum_k\frac1{k^2\psi(k)}}<\infty  \hspace{3mm} (\text{resp.}=\infty), $$then for almost all (resp. almost no) pairs $Y\in M_{m\times n}$, $\vb\in \R^{{m}}$, the system $$ \|Y\vq +\vb - \vp \|^{m} < \psi(t) \hspace{10mm} \|\vq\|^{n}<t $$ is solvable in integer vectors  $\vq  \in \Z^n$ and $\vp \in \Z^m$ for sufficiently large $t$.  \end{thm}

\end{document}